\documentclass[12pt]{article}
\input epsf.tex

\usepackage{amsmath}
\usepackage{amsthm}
\usepackage{amsfonts}
\usepackage{amssymb}
\usepackage{graphicx}
\usepackage{latexsym}

\usepackage{amsmath,amsthm,amsfonts,amssymb}

\usepackage{epsfig}

\usepackage[]{xy}

\theoremstyle{plain}
\newtheorem{thm}{Theorem}[section]
\newtheorem{prop}[thm]{Proposition}
\newtheorem{lem}[thm]{Lemma}
\newtheorem{cor}[thm]{Corollary}

\theoremstyle{definition}
\newtheorem{defn}{Definition}
\theoremstyle{remark}

\advance \topmargin by -\headheight
\advance \topmargin by -\headsep
\evensidemargin \oddsidemargin
\def\B{{\textrm{B}}}

\def\denom{{\textrm{q}}}

\def\bE{{\mathbb E}}

\def\bphi{{\bar \phi}}

\def\cG{{\mathcal G}}

\def\N{{\mathbb N}}

\def\R{{\mathbb R}}

\def\sym{{\textrm{Sym}}}

\def\weight{\mathcal{W}}

\def\chix{{\raise.5ex\hbox{$\chi$}}}

\def\Z{{\mathbb Z}}

\begin{document}
\title{The ergodic theory of free group actions:\\ entropy and the $f$-invariant}
\author{Lewis Bowen \\ University of Hawai'i, M\=anoa\\ lpbowen@math.hawaii.edu}
\begin{abstract}
Previous work introduced two measure-conjugacy invariants: the $f$-invariant (for actions of free groups) and $\Sigma$-entropy (for actions of sofic groups). The purpose of this paper is to show that the $f$-invariant is essentially a special case of $\Sigma$-entropy. There are two applications: the $f$-invariant is invariant under group automorphisms and there is a uniform lower bound on the $f$-invariant of a factor in terms of the original system. 
\end{abstract}
\maketitle

\noindent

\section{Introduction}
The paper [Bo08b] introduced a measure-conjugacy invariant, called $\Sigma$-entropy, for measure-preserving actions of a sofic group. This was applied, for example, to classify Bernoulli shifts over an arbitrary countable linear group. Previously, [Bo08a] introduced the $f$-invariant for measure-preserving actions of free groups. The invariants of both papers have strong analogies with classical Kolmogorov-Sinai entropy. The purpose of this paper is to show that the $f$-invariant is essentially a special case of $\Sigma$-entropy. We apply this result to show the $f$-invariant does not change under group automorphisms and that there is a lower bound on the $f$-invariant of a factor in terms of the $f$-invariant of the system. The introductions to [Bo08a-b] provide further background and motivation for $\Sigma$-entropy and the $f$-invariant.



To define $\Sigma$-entropy precisely, let $G$ be a countable group and let $\Sigma=\{\sigma_i\}_{i=1}^\infty$ be a sequence of homomorphisms $\sigma_i:G \to \sym(m_i)$ where $\sym(m_i)$ denotes the full symmetric group of the set $\{1,\ldots,m_i\}$. $\Sigma$ is {\em asymptotically free} if for every pair $g_1,g_2\in G$ with $g_1\ne g_2$
$$\lim_{i\to\infty} \frac{\big| \{ 1\le j \le m_i ~:~ \sigma_i(g_1)j=\sigma_i(g_2)j\} \big| }{m_i} = 0.$$
 The treatment of $\Sigma$-entropy given next differs from [Bo08b] in two respects: for simplicity, we assume that each $\sigma_i$ is a homomorphism and we use observables rather than partitions to define it.

We will write $G \curvearrowright^T (X,\mu)$ to mean $(X,\mu)$ is a standard probability measure space and $T=(T_g)_{g\in G}$ is an action of $G$ on $(X,\mu)$ by m.p. (measure preserving) transformations. This means that for each $g\in G$, $T_g:X \to X$ is a m.p. transformation and $T_{g_1}T_{g_2}=T_{g_1g_2}$. An {\em observable} of $(X,\mu)$ is a measurable map $\phi:X \to A$ where $A$ is a finite or countably infinite set. We will say that $\phi$ is {\em finite} if $A$ is finite. Roughly speaking, the $\Sigma$-entropy rate of $\phi$ is the exponential rate of growth of the number of observables $\psi:\{1,\ldots,m_i\} \to A$ that approximate $\phi$. In order to make precise what it means to approximate, we need to introduce some definitions.


If $\phi: X \to A$ and $\psi: X \to B$ are two observables, then the {\em join} of $\phi$ and $\psi$ is the observable $\phi \vee \psi: X \to A \times B$ defined by $\phi \vee \psi(x)=\big( \phi(x),\psi(x)\big)$. If $g\in G$ then $T_{g}\phi: X\to A$ is defined by $T_{g}\phi(x)=\phi(T_gx)$.  If $H \subset G$ is finite, then let $\phi^H := \bigvee_{h \in H} T_{h}\phi.$ $\phi^H$ maps $X$ into $A^H$, the direct product of $|H|$ copies of $A$. Let $\phi^H_*\mu$ denote the pushforward of $\mu$ on $A^H$. In other words, $\phi^H_*(\mu)(S)=\mu\big((\phi^H)^{-1}(S)\big)$ for $S \subset A^H$. 

For each $i$, let $\zeta_i$ denote the uniform probability measure on $\{1,\ldots,m_i\}$. If $\psi:\{1,\ldots,m_i\} \to A$ is an observable and $H\subset G$ then let $\psi^H:= \bigvee_{h \in H} \sigma_i(h)\psi$ where $\sigma_i(h)\psi:\{1,\ldots, m_i\} \to A$ is defined by  $\sigma_i(h)\psi(j)=\psi(\sigma_i(h)j)$. Of course, $\psi^H$ depends on $\sigma_i$ but, to keep the notation simple, we will leave this dependence implicit. Let $\psi^H_*\zeta_i$ be the pushforward of $\zeta_i$ on $A^H$. Finally, let $d^H_{\sigma_i}(\phi,\psi)$ be the $l^1$-distance between $\phi^H_*\mu$ and $\psi^H_*\zeta_i$. In other words,
$$d^H_{\sigma_i}(\phi,\psi) = \sum_{a \in A^H} \big| \phi^H_*\mu(a) - \psi^H_*\zeta_i(a) \big|.$$

\begin{defn}
If $\phi:X \to A$ is an observable and $A$ is finite then define the {\em $\Sigma$-entropy rate} of $\phi$ by
$$h(\Sigma,T,\phi):=\inf_{H \subset G} \inf_{\epsilon>0} \limsup_{i\to\infty} \frac{1}{m_i}\log\Big( \big| \{\psi:\{1,\ldots,m_i\}\to A~:~d^H_{\sigma_i}(\phi,\psi)\le \epsilon\}\big|\Big).$$
The first infimum above is over all finite subsets $H\subset G$.
\end{defn}
\begin{defn}
Define the {\em entropy} of $\phi$ by
$$H(\phi):=-\sum_{a \in A} \mu\big(\phi^{-1}(a)\big) \log \Big(\mu\big(\phi^{-1}(a)\big)\Big).$$
\end{defn}
\begin{defn}\label{defn:infinite}
If $\phi:X \to A$ is an observable and $A$ is countably infinite then let $\pi_n:A \to A_n$ be a sequence of maps such that
\begin{enumerate}
\item $A_n$ is a finite set for all $n$;
\item for each $i>j$ there is a map $\pi_{ij}:A_i \to A_j$ such that $\pi_j = \pi_{ij} \circ \pi_i$;
\item $\pi_n$ is asymptotically injective in the sense that for all $a,b \in A$ with $a\ne b$ there exists $N$ such that $n>N$ implies $\pi_n(a) \ne \pi_n(b)$.
\end{enumerate}
Now define
$$h(\Sigma,T,\phi):= \lim_{n\to\infty} h(\Sigma,T,\pi_n \circ \phi).$$
In [Bo08b] it is proven that if $H(\phi)<\infty$ then this limit exists and is independent of the choice of sequence $\{\pi_n\}$.
\end{defn}

An observable $\phi$ is {\em generating} if the smallest $G$-invariant $\sigma$-algebra on $X$ that contains $\{\phi^{-1}(a)\}_{a\in A}$ is equal to the $\sigma$-algebra of all measurable sets up to sets of measure zero. The next theorem is (part of) the main result of [Bo08b]. 

\begin{thm}\label{thm:Bo08b}
Let $\Sigma=\{\sigma_i\}$ be an asymptotically free sequence of homomorphisms $\sigma_i:G \to \sym(m_i)$ for a group $G$. Let $G \curvearrowright^T (X,\mu)$. If $\phi_1$ and $\phi_2$ are two finite-entropy generating observables then $h(\Sigma,T, \phi_1)=h(\Sigma,T, \phi_2)$. 
\end{thm}
This motivates the following definition.
\begin{defn}
If $\Sigma$ and $T$ are as above then the $\Sigma$-entropy of the action $T$ is defined by $h(\Sigma,T):=h(\Sigma,\phi)$ where $\phi$ is any finite-entropy generating observable (if one exists). 
\end{defn}

Next let us discuss a slight variation on $\Sigma$-entropy. Let $\{m_i\}_{i=1}^\infty$ be a sequence of natural numbers. For each $i\in \N$, let $\mu_i$ be a probability measure on the set of homomorphisms from $G$ to $\sym(m_i)$. Let $\sigma_i:G\to \sym(m_i)$ be chosen at random according to $\mu_i$. The sequence $\Sigma=\{\mu_i\}_{i=1}^\infty$ is said to be {\em asymptotically free} if for every pair $g_1,g_2\in G$ with $g_1\ne g_2$,
$$\lim_{i\to\infty} \frac{\bE\Big[\big| \{ 1\le j \le m_i ~:~ \sigma_i(g_1)j=\sigma_i(g_2)j\} \big| \Big]}{m_i} = 0$$
where $\bE[\cdot]$ denotes expected value. The $\Sigma$-entropy rate of an observable $\phi:X\to A$ with $A$ finite  is defined by
$$h(\Sigma,T,\phi):=\inf_{H \subset G} \inf_{\epsilon>0} \limsup_{i\to\infty}\frac{1}{m_i}\log\Big( \bE\Big[\big| \{\psi:\{1,\ldots,m_i\}\to A~:~d^H_{\sigma_i}(\phi,\psi)\le \epsilon\}\big|\Big]\Big).$$
With these definitions in mind, theorem \ref{thm:Bo08b} is still true if ``homomorphisms'' is replaced with ``probability measures on the set of homorphisms''.

Let us note one more generalization. If $G$ is a semigroup with identity then the above definitions still make sense. Using results from [Bo08c] it can be shown that theorem \ref{thm:Bo08b} remains true.

Now let us recall the $f$-invariant from [Bo08a]. Let $G=\langle s_1,\ldots, s_r \rangle$ be either a free group or free semigroup of rank $r$. Let $G \curvearrowright^T (X,\mu)$. Let $\alpha$ be a partition of $X$ into at most countably many measurable sets. The {\em entropy} of $\alpha$ is defined by 
$$H(\alpha):=-\sum_{A\in\alpha} \mu(A)\log(\mu(A))$$
where, by convention, $0\log(0)=0$. If $\alpha$ and $\beta$ are partitions of $X$ then the {\em join} is the partition $\alpha \vee \beta:=\{A \cap B~|~A\in \alpha, B \in\beta\}$. Let $B(e,n)$ denote the ball of radius $n$ in $G$ with respect to the word metric induced by its generating set (which is either $\{s_1,\ldots, s_r\}$ if $G$ is a semigroup or $\{s_1^{\pm 1}, \ldots, s_r^{\pm 1}\}$ is $G$ is a group). Define
\begin{eqnarray*}
F(T,\alpha)&:=&(1-2r)H(\alpha) + \sum_{i=1}^r H(\alpha \vee T_{s_i}^{-1}\alpha)\\
\alpha^n&:=&\bigvee_{g\in B(e,n)} T_g^{-1}\alpha\\
f(T,\alpha)&:=&\inf_n F(T,\alpha^n).
\end{eqnarray*}
The partition $\alpha$ is {\em generating} if the smallest $G$-invariant $\sigma$-algebra containing $\alpha$ equals the $\sigma$-algebra of all measurable sets up to sets of measure zero.

\begin{thm} 
Let $G=\langle s_1,\ldots,s_r\rangle$ be a free group or free semigroup.  Let $G \curvearrowright^T (X,\mu)$. If $\alpha_1$ and $\alpha_2$ are two generating partitions with $H(\alpha_1)+H(\alpha_2)<\infty$ then $f(T,\alpha_1)=f(T,\alpha_2)$.
\end{thm}
This theorem was proven in [Bo08c]. The special case in which $G$ is a group and $\alpha_1,\alpha_2$ are finite is the main result of [Bo08a]. Because of this theorem, we define the $f$-invariant of the action by $f(T):=f(T,\alpha)$ where $\alpha$ is any finite-entropy generating partition of $X$ (if one exists). 

In order to relate this result with $\Sigma$-entropy, let us make the following definitions. If $\phi:X \to A$ is an observable, then let $\bphi = \{\phi^{-1}(a)\}_{a\in A}$ be the corresponding partition of $X$. Define $F(T,\phi):=F(T,\bphi)$ and $f(T,\phi):=f(T,\bphi)$. The main result of this paper is:
\begin{thm}\label{thm:main}
Let $G=\langle s_1,\ldots, s_r\rangle$ be a free group or free semigroup of rank $r\ge 1$.  Let $G \curvearrowright^T (X,\mu)$. Let $\phi$ be a finite observable. For $i \ge 1$, let $\mu_i$ be the uniform probability measure on the set of all homomorphisms from $G$ to $\sym(i)$. Let $\Sigma=\{\mu_i\}_{i=1}^\infty$. Then $h(\Sigma,T,\phi) = f(T,\phi).$
\end{thm}

We will prove a refined version of this theorem as follows. Recall the definition of $d^H_{\sigma_i}(\phi,\psi)$ given above. Define
$$d^*_{\sigma_i}(\phi,\psi) := \sum_{i=1}^r d^{ \{e,s_i\} }_{\sigma_i}(\phi,\psi).$$
\begin{thm}\label{thm:second}
Let $G$ and $T$ be as in the previous theorem. Let $\phi:X \to A$ be a finite observable. Let $\sigma_i:G\to \sym(i)$ be a homomorphism chosen uniformly at random. Then
$$F(T,\phi) = \inf_{\epsilon>0} \lim_{i\to\infty}\frac{1}{i}\log\Big( \bE\Big[\big| \{\psi:\{1,\ldots,i\}\to A~:~d^*_{\sigma_i}(\phi,\psi)\le \epsilon\}\big|\Big]\Big). $$
\end{thm}
This theorem is proven in section \ref{sec2}. In section \ref{sec3} we prove theorem \ref{thm:main} from it.

\subsection{Application I: automorphism invariance}

Let $G$ be a countable group or semigroup. Let $G \curvearrowright^T (X,\mu)$. Let $\omega:G \to G$ be an automorphism. Let $T^\omega=(T^\omega_g)_{g\in G}$ where $T^\omega_gx:=T_{\omega(g)}x$ for all $x\in X$. This new action of $G$ is not necessarily isomorphic to the original action. That is, there might not exist a map $\phi:X \to X$ such that $\phi(T_gx)=T^\omega_g\phi(x)$ for a.e. $x\in X$ and all $g\in G$.

Let $\Sigma=\{\sigma_i\}$ be an asymptotically free sequence of homomorphisms $\sigma_i:G \to \sym(m_i)$. Let $\Sigma^\omega=\{\sigma_i \circ \omega\}$. A short exercise reveals that $h(\Sigma,T,\phi)=h(\Sigma^\omega,T^\omega,\phi)$ for any  $\phi$. 

If $\sigma_i:G \to \sym(i)$ is chosen uniformly at random, it follows that the law of $\sigma_i \circ \omega$ is the same as the law of $\sigma_i$. Therefore, if $\mu_i$ is the uniform probability measure on the set of homomorphisms from $G$ to $\sym(i)$ and $\Sigma=\{\mu_i\}$, then $h(\Sigma,T,\phi) = h(\Sigma,T^\omega,\phi)$. Theorem \ref{thm:main} now implies:
\begin{thm}
Let $G$ and $T$ be as in theorem \ref{thm:main}. Let $\omega:G \to G$ be an automorphism. Then for any finite observable $\phi$, $f(T,\phi)=f(T^\omega,\phi)$.
\end{thm}
This implies that $f(T,\phi)$ does not depend on the choice of free generator set $\{s_1,\ldots, s_r\}$ for $G$ since any two free generating sets are related by an automorphism.

\subsection{Application II: lower bounds on the $f$-invariant of a factor}

\begin{defn}
 Let $G \curvearrowright^T (X,\mu)$ and $G \curvearrowright^S (Y,\nu)$. Then $S$ is a {\em factor} of $T$ if there exists a measurable map $\phi:X \to Y$ such that $\phi_*\mu=\nu$ and $\phi(T_gx)=S_g\phi(x)$ for all $g\in G$ and a.e. $x\in X$.
\end{defn}

To motivate this section, let us point out two curious facts. 

First, Ornstein proved in [Or70] that every factor of a Bernoulli shift over $\Z$ is measurably conjugate to a Bernoulli shift. It is not known whether this holds when $\Z$ is replaced with a nonabelian free group. A counterexample due to Sorin Popa [Po08] (based on [PS07]) shows that if $G$ is an infinite property $T$ group then there exists a factor of a Bernoulli shift over $G$ that is not measurably conjugate to a Bernoulli shift.

Second, the $f$-invariant of an action can be negative. For example, if $X$ is a set with $n$ elements, $\mu$ is the uniform measure on $X$ and $T=(T_g)_{g\in G}$ is a m.p. action of $G=\langle s_1,\ldots, s_r \rangle$ on $X$ then $f(T) = -(r-1)\log(n)$.

From these two facts a natural question arises: can the $f$-invariant of a factor of a Bernoulli shift over $G$ be negative? To answer this, let us recall the following result from [Bo08b, corollary 8.3].


\begin{lem}
Let $G$ be a countable group. Let $\Sigma=\{\sigma_i\}_{i=1}^\infty$ be an asymptotically free sequence of homomorphisms $\sigma_i:G \to \sym(m_i)$. Let $T$ be a m.p. action of $G$ and let $S$ be a factor of $T$. Assume that there exist finite-entropy generating partitions for $T$ and $S$. Also let $\phi$ be a generating observable for $T$ with $H(\phi)<\infty$. Then
$$h(\Sigma,S) \ge h(\Sigma, T) - H(\phi).$$
\end{lem}
So theorem \ref{thm:main} implies:
\begin{thm}
Let $G=\langle s_1,\ldots, s_r\rangle$ be a free group on $r$ generators. Let $T$ be a m.p. action of $G$ and let $S$ be a factor of $T$. Assume there exists finite generating partitions for $T$ and $S$. Let $\alpha$ be a finite generating partition for $T$. Then
$$f(S) \ge f(T) - H(\alpha).$$
\end{thm}
In order to apply this to Bernoulli shifts, let us recall the definitions. Let $K$ be a finite or countable set and $\kappa$ a probability measure on $K$. Let $(K^G,\kappa^G)$ denote the product measure space. Define $T_g:K^G \to K^G$ by $T_g(x)(h)=x(hg)$. This defines a measure-preserving action of $G$ on $(K^G,\kappa^G)$. It is the {\em Bernoulli shift} over $G$ with base measure $\kappa$. In [Bo08a] it was shown that $f(T)=H(\kappa)$ where 
$$H(\kappa):=-\sum_{k\in K} \mu(\{k\}) \log(\mu(\{k\}).$$

Let $\alpha$ be the canonical partition of $K^G$. I.e., $\alpha=\{A_k~:~k\in K\}$ where $A_k=\{x \in K^G~:~x(e)=k\}$. Note $H(\alpha)=H(\kappa)=f(T)$. So the theorem above implies
\begin{cor}
If $S$ is a factor of the Bernoulli shift and if there exists a finite generating partition for $S$ then $f(S) \ge 0$.
\end{cor}

It is unknown whether there exists a nontrivial factor $S$ of a Bernoulli shift over a free group $G$ such that $f(S)=0$.

In [Bo08c], classical Markov chains are generalized to Markov chains over free groups. An explicit example was given of a Markov chain with finite negative $f$-invariant. It follows that this Markov chain cannot be measurably conjugate to a factor of a Bernoulli shift. It can be shown that this Markov chain is uniformly mixing. To contrast this with the classical case, recall that Friedman and Ornstein proved in [FO70] that every mixing Markov chain over the integers is isomorphic to a Bernoulli shift. 

Now we can construct a mixing Markov chain with positive $f$-invariant that is not isomorphic to a Bernoulli shift as follows. Let $T$ denote a mixing Markov chain with negative $f$-invariant. Let $S$ denote a Bernoulli shift with $f(S)>-f(T)$. Consider the product action $T \times S$. A short computation reveals that, in general, $f(T\times S) = f(T) + f(S)$. Therefore $T\times S$ has positive $f$-invariant. It can be shown that $T\times S$ is a mixing Markov chain. However it cannot be isomorphic to a Bernoulli shift since it factors onto $T$ which has negative $f$-invariant.





\section{Proof of theorem  \ref{thm:second}}\label{sec2}
Let $G=\langle s_1,\ldots, s_r\rangle$ be a free group or free semigroup of rank $r$. Let $G\curvearrowright^T (X,\mu)$.  Let $\phi:X \to A$ be a finite observable.

We will need to consider certain perturbations of the measure $\mu$ with respect to the given observable $\phi:X \to A$. For this purpose we introduce the notion of weights on the graph $\cG=(V,E)$ that is defined as follows. The vertex set $V$ equals $A$. For every $a,b\in A$ and every $i \in \{1,\ldots, r\}$ there is a {\em directed} edge from $a$ to $b$ labeled $i$. This edge is denoted $(a,b;i)$. We allow the possibility that $a=b$. A {\em weight} on $\cG$ is a function $W:V \sqcup E \to [0,1]$ satisfying:
\begin{eqnarray*}
W(a) = \sum_{b \in A} W(a,b;i) &=& \sum_{b \in A} W(b,a;i) ~\forall i=1\ldots r, ~\forall a\in A,\\
1 &=& \sum_{a \in A} W(a).
\end{eqnarray*}
For example, 
\begin{eqnarray*}
W_\mu(a)&:=&\mu(\phi^{-1}(a)),\\
W_\mu(a,b;i)&:=& \mu\big( \{x \in X~:~ \phi(x)=a, \phi(T_{s_i}x) = b \}\big)
\end{eqnarray*}
is the weight associated to $\mu$. If $\sigma:G \to \sym(n)$ is a homomorphism and $\psi:\{1,\ldots,n\} \to A$ is a function then we define the weight $W_{\sigma,\psi}$ by
\begin{eqnarray*}
W_{\sigma,\psi}(a)&:=&|\psi^{-1}(a)|/n,\\
W_{\sigma,\psi}(a,b;i)&:=& \Big|\big\{ j~:~ \psi(j)=a, \psi\big(\sigma(s_i)j\big) = b \big\}\Big|/n.
\end{eqnarray*}
Note that 
$$d^*_\sigma(\phi,\psi) = \sum_{i=1}^r \sum_{a,b\in A} \big|W_\mu(a,b;i) - W_{\sigma,\psi}(a,b;i)\big|.$$
So given two weights $W_1$, $W_2$ define
$$d_*(W_1,W_2) := \sum_{i=1}^r \sum_{a,b\in A} \big|W_1(a,b;i) - W_2(a,b;i)\big|.$$


\begin{prop}
Let $n$ be a positive integer. Let $W$ be a weight. Suppose that $W(a,b;i)n \in \Z$ for every $a,b \in A$ and every $i=1\ldots r$. 
If $\sigma:G \to \sym(n)$ is chosen uniformly at random then
\begin{eqnarray*}
\bE\Big[\big| \{\psi:\{1,\ldots,n\}\to A~:~d_*(W,W_{\sigma,\psi})=0\}\big|\Big] =  \frac{n!^{1-r} \prod_{a\in A} (nW(a))!^{2r-1}}{ \prod_{i=1}^r \prod_{a,b \in A} (nW(a,b;i))!}.
\end{eqnarray*}
\end{prop}

\begin{proof}
Note that if $d_*(W,W_{\sigma,\psi})=0$ then for all $a\in A$, $W_{\sigma,\psi}(a)=W(a)$. Equivalently,
\begin{eqnarray}\label{eqn1}
|\psi^{-1}(a)| = n W(a) ~\forall ~ a \in A.
\end{eqnarray}

The number of functions $\psi:\{1,\ldots, n\} \to A$ that satisfy this requirement is
$$\frac{n!}{ \prod_{a \in A} (n W(a))!}.$$
If $\psi_1, \psi_2$ are two different functions that satisfy equation \ref{eqn1} then there is a permutation $\tau \in \sym(n)$ such that $\psi_1=\psi_2 \circ \tau$. If $\sigma^\tau: G\to \sym(n)$ is the homomorphism defined by $\sigma^\tau(g)=\tau\sigma(g)\tau^{-1}$ then  $W_{\sigma,\psi_1}=W_{\sigma^\tau,\psi_2}$. Since $\sigma: G\to \sym(n)$ is chosen uniformly at random, this implies that the probability that $d_*(W,W_{\sigma,\psi_1})=0$ is the same as the probability that $d_*(W,W_{\sigma,\psi_2})=0$. So fix a particular function $\psi_0$ satisfying equation \ref{eqn1}. Then

\begin{eqnarray*}
\bE\Big[\big| \{\psi:\{1,\ldots,n\}\to A~:~d_*(W,W_{\sigma,\psi})=0\}\big|\Big]=\frac{n!\textrm{Prob}[d_*(W,W_{\sigma,\psi_0})=0]}{ \prod_{a \in A} (n W(a))! }.
\end{eqnarray*}

For any two weights $W_1,W_2$ and $1\le i \le r$, define
$$d_i(W_1,W_2) := \sum_{a,b \in A} \big|W_1(a,b;i) - W_2(a,b;i)\big|.$$
So $d_*=\sum_{i=1}^r d_i$.  

The homomorphism $\sigma:G \to \sym(n)$ is determined by its values $\sigma(s_1),\ldots, \sigma(s_r)$. The event $d_i(W,W_{\sigma,\psi_0})=0$ is determined by $\sigma(s_i)$. So if $i\ne j$ then the events $d_i(W,W_{\sigma,\psi_0})=0$ and $d_j(W,W_{\sigma,\psi_0})=0$ are independent. Therefore,
\begin{eqnarray}\label{eqn2}
\bE\Big[\big| \{\psi:\{1,\ldots,n\}\to A~:~d_*(W,W_{\sigma,\psi})=0\}\big|\Big]=\frac{n!\prod_{i=1}^r \textrm{Prob}[d_i(W,W_{\sigma,\psi_0})=0]}{ \prod_{a \in A} (n W(a))!}.
\end{eqnarray}

Fix $i\in \{1,\ldots, r\}$. We will compute $\textrm{Prob}[d_i(W,W_{\sigma,\psi_0})=0]$. The element $\sigma(s_i)$ induces a pair of partitions $\alpha,\beta$ of $\{1,\ldots,n\}$ as follows. $\alpha:=\{P_{a,b}~|~a,b \in A\}$ and $\beta:=\{Q_{a,b}~|~a,b \in A\}$ where 
\begin{eqnarray*}
P_{a,b}&=&\{j~:~\psi_0(j)=a \textrm{ and } \psi_0(\sigma(s_i)j)=b\}\\
Q_{a,b}&=&\{j~:~\psi_0(j)=b \textrm{ and } \psi_0(\sigma(s_i)^{-1}j)=a\}.
\end{eqnarray*}
Also there is a bijection from $M_{a,b}:P_{a,b} \to Q_{a,b}$ defined by $M_{a,b}(j)=\sigma(s_i)j$. Conversely, $\sigma(s_i)$ is uniquely determined by these partitions and bijections.

Note that $|P_{a,b}|=|Q_{a,b}|=nW_{\sigma,\psi_0}(a,b;i)$. Thus $d_i(W,W_{\sigma,\psi_0})=0$ if and only $|P_{a,b}|=|Q_{a,b}| = nW(a,b;i)$ for all $a,b\in A$. If this occurs then $|\cup_{b\in A} P_{a,b}| = nW(a)$ for all $a\in A$. So the number of pairs of partitions $\alpha, \beta$ that satisfy this requirement is 
$$\frac{\prod_{a\in A} (nW(a))!^2}{\prod_{a,b \in A} \big(nW(a,b;i))!\big)^{2}}.$$

Given such a pair of partitions, the number of collections of bijections $M_{a,b}:P_{a,b}\to Q_{a,b}$ (for $a,b\in A$) equals $\prod_{a,b\in A} (nW( a, b ;i))!.$ Since there are $n!$ elements in $\sym(n)$ it follows that
$$\textrm{Prob}[d_i(W,W_{\sigma,\psi_0})=0]=\frac{\prod_{a\in A} (nW(a))!^2}{n! \prod_{a,b \in A} (nW(a,b;i))!}.$$
The proposition now follows from this equality and equation \ref{eqn2}.
\end{proof}

Let $\weight$ be the set of all weights on $\cG$. It is a compact convex subset of $\R^d$ for some $d>0$. Define $F:\weight \to \R$ by 
$$F(W):=-\Big(\sum_{i=1}^r \sum_{a,b\in A} W(a,b;i)\log(W(a,b;i))\Big) + (2r-1)\sum_{a\in A} W(a)\log(W(a)).$$
We follow the usual convention that $0\log(0)=0$. Observe that $F(T,\phi)=F(W_\mu)$. 

Given a weight $W$, let $\denom_W$ denote the smallest positive integer such that $W(a,b;i)\denom_W \in \Z$ for all $a,b\in A$ and for all $i\in \{1,\ldots,r\}$. If no such integer exists then set $\denom_W:=+\infty$. If $p$ and $q$ are integers, $p\ne 0$ and $\frac{q}{p} \in \Z$ then we write $p \mid q$. Otherwise we write $p \nmid q$.

\begin{lem}\label{lem:continuous}
$F:\weight \to \R$ is continuous. Also, there exist constants $0<c_1<c_2$ and $p_1<p_2$ such that for every weight $W$ with $\denom_W<\infty$ and every $n \ge 1$ such that $\denom_W \mid n$, if $\sigma:G \to \sym(n)$ is chosen uniformly at random then
\begin{eqnarray*}
c_1n^{p_1} e^{F(W)n} \le  \bE\Big[\Big| \{\psi:\{1,\ldots,n\}\to A~:~d_*(W,W_{\sigma,\psi})=0\}\Big|\Big] \le c_2n^{p_2} e^{F(W)n}.
\end{eqnarray*}
\end{lem}

\begin{proof}
It is obvious that $F$ is continuous. The second statement follows from the previous proposition and Stirling's approximation. The constants depend only on $|A|$ and the rank $r$ of $G$.
\end{proof}

\begin{lem}
There exists a constant $k>0$ such that the following holds. Let $W$ be a weight and let $n>0$ be a positive integer. Then there exists a weight $\widetilde{W}$ such that $\denom_{\widetilde{W}}<\infty$, $\denom_{\widetilde{W}} |n $ and $d_*(W,\widetilde{W})<k/n$.  
\end{lem}

\begin{proof}

Choose $a_0 \in A$. For $b, c \in A-\{a_0\}$ and $i\in \{1,\ldots, r\}$ define

\begin{eqnarray*}
\widetilde{W}(b) &:=& \frac{\lfloor W(b)n\rfloor }{n.}\\
\widetilde{W}(a_0)&:=&1- \sum_{b \in A-\{a_0\}} \widetilde{W}(b).\\
\widetilde{W}(b,c;i) &:=& \frac{ \lfloor W(b,c;i)n\rfloor }{n}.\\
\widetilde{W}(a_0,b;i) &:=& \widetilde{W}(b) - \sum_{a \in A-\{a_0\}} \widetilde{W}(a,b;i).\\
\widetilde{W}(b, a_0;i) &:=& \widetilde{W}(b) - \sum_{a \in A-\{a_0\}} \widetilde{W}(b,a;i).\\
\widetilde{W}(a_0,a_0;i) &:=& \widetilde{W}(a_0) - \sum_{b\in A-\{a_0\}} \widetilde{W}(a_0,b;i).
\end{eqnarray*}
Let us check that $\widetilde{W}$ is a weight. It is clear that $\sum_{a\in A} \widetilde{W}(a)=1$. If $b \in A-\{a_0\}$ then $\widetilde{W}(b)=\sum_{a\in A} \widetilde{W}(a,b;i) = \sum_{a\in A} \widetilde{W}(b,a;i).$ It is immediate that $\widetilde{W}(a_0)= \sum_{b\in A} \widetilde{W}(a_0,b;i).$ Also
\begin{eqnarray*}
\sum_{b\in A} \widetilde{W}(b,a_0;i) &=& \widetilde{W}(a_0,a_0;i)+\sum_{b\in A-\{a_0\}} \widetilde{W}(b,a_0;i)\\
&=& \widetilde{W}(a_0) - \sum_{b\in A-\{a_0\}} \widetilde{W}(a_0,b;i) + \sum_{b\in A-\{a_0\}} \widetilde{W}(b,a_0;i)\\
&=&\widetilde{W}(a_0) + \sum_{b\in A-\{a_0\}}  \widetilde{W}(b,a_0;i) - \widetilde{W}(a_0,b;i)\\
&=&\widetilde{W}(a_0) +  \sum_{b\in A-\{a_0\}}\Big( \widetilde{W}(b) - \sum_{a \in A-\{a_0\}} \widetilde{W}(b,a;i)\Big) - \Big( \widetilde{W}(b) - \sum_{a \in A-\{a_0\}} \widetilde{W}(a,b;i)\Big)\\
&=&\widetilde{W}(a_0).
\end{eqnarray*}
This proves that $\widetilde{W}$ is a weight. It is clear that $\denom_{\widetilde{W}}<\infty$ and $\denom_{\widetilde{W}} | n$. Lastly observe that if $a,b \in A-\{a_0\}$ then $|W(a,b;i)-\widetilde{W}(a,b;i)| \le 1/n$. Since $|W(b)-\widetilde{W}(b)|\le 1/n$ too, $|W(a_0,b;i)-\widetilde{W}(a_0,b;i)| \le |A|/n$ and  $|W(b,a_0;i)-\widetilde{W}(b,a_0;i)| \le |A|/n$. Since $|W(a_0)-\widetilde{W}(a_0)|\le |A|/n$, $|W(a_0,a_0;i) - \widetilde{W}(a_0,a_0;i)| \le |A|^2/n$. Thus $d_*(W,\widetilde{W}) \le r|A|^2/n$. 
\end{proof}

We are now ready to prove theorem \ref{thm:second}.



\begin{proof}[Proof of theorem \ref{thm:second}]
 Recall that $\phi:X \to A$ is an observable and $A$ is a finite set. Let $n\ge 0$ and let $\sigma_n:G \to \sym(n)$ be a homomorphism chosen uniformly at random. Given a weight $W$, let 
$$Z_n(W):=\big| \{\psi:\{1,\ldots,n\} \to A~:~ d_*(W_{\sigma_n,\psi}, W) = 0\}\big|.$$
For any $\epsilon>0$,
\begin{eqnarray}\label{eqn:what}
\bE\Big[\big| \{\psi:\{1,\ldots,n\}\to A~:~d^*_{\sigma_n}(\phi,\psi)\le \epsilon\}\big|\Big] = \sum_{W:~d_*(W,W_\mu)\le\epsilon} \bE[Z_n(W)].
\end{eqnarray}
Let $\delta>0$. Since $F:\weight \to \R$ is continuous, there exists $\epsilon_0>0$ such that if $d_*(W,W_\mu)\le\epsilon_0$ then $|F(W)-F(W_\mu)| <\delta$. So let us fix $\epsilon$ with $0<\epsilon<\epsilon_0$.

By the previous lemma, if $n$ is sufficiently large then there exists a weight $W$ such that $d_*(W,W_\mu) \le \epsilon$ and $\denom_W \mid n$. Lemma \ref{lem:continuous} implies
\begin{eqnarray}\label{eqn:a1}
\bE\Big[\big| \{\psi:\{1,\ldots,n\}\to A~:~d^*_{\sigma_n}(\phi,\psi)\le \epsilon\}\big|\Big] \ge \bE[Z_n(W)] \ge c_1n^{p_1}e^{F(W_\mu)n-\delta n}
\end{eqnarray}
where $c_1>0$ and $p_1$ are constants. 

If $W$ is a weight such that $\denom_W \nmid n$ then $Z_n(W)=0$. If $\denom_W \mid n$ then $W(a,b;i) \in \Z[1/n]$ for all $a,b\in A$ and $i\in\{1,\ldots,r\}$. The space of all weights lies inside the cube $[0,1]^d \subset \R^d$ for some $d$. So the number of weights $W$ such that $Z_n(W)\ne 0$ is at most $n^d$. Lemma \ref{lem:continuous} and equation \ref{eqn:what} now imply that
\begin{eqnarray}\label{eqn:a2}
\bE\Big[\big| \{\psi:\{1,\ldots,n\}\to A~:~d^*_{\sigma_n}(\phi,\psi)\le \epsilon\}\big|\Big] \le c_2n^{p_2+d}e^{F(W_\mu)n+\delta n}.
\end{eqnarray}
Here $c_2>0$ and $p_2$ are constants. Equations \ref{eqn:a1} and \ref{eqn:a2} imply
$$\limsup_{n\to\infty} \Big|\frac{1}{n}\log\Big(\bE\Big[\big| \{\psi:\{1,\ldots,n\}\to A~:~d^*_{\sigma_n}(\phi,\psi)\le \epsilon\}\big|\Big]\Big)-F(W_\mu)\Big|\le \delta.$$
Since $\delta$ is arbitrary, it follows that
$$\inf_{\epsilon >0} \lim_{n\to\infty} \frac{1}{n}\log\Big(\bE\Big[\big| \{\psi:\{1,\ldots,n\}\to A~:~d^*_{\sigma_n}(\phi,\psi)\le \epsilon\}\big|\Big]\Big)=F(W_\mu)=F(T,\phi).$$

\end{proof}

\section{Proof of theorem \ref{thm:main}}\label{sec3}

As in the statement of theorem \ref{thm:main}, let $G=\langle s_1,\ldots, s_r\rangle$ be a free group or free semigroup of rank $r\ge 1$. Let $G \curvearrowright^T (X,\mu)$. Let $\phi:X \to A$ be a finite observable. Let $\Sigma=\{\mu_i\}_{i=1}^\infty$ where each $\mu_i$ is the uniform probability measure on the set of homomorphisms from $G$ to $\sym(i)$. Let $\sigma_i:G \to \sym(i)$ be a homomorphism chosen uniformly at random among all homomorphisms of $G$ into $\sym(i)$. Theorem \ref{thm:main} is an immediate consequence of the next two propositions.
\begin{prop}
$$h(\Sigma,T,\phi) \le f(T,\phi).$$
\end{prop}

\begin{proof}
 Let $S=\{s_1,\ldots,s_r\}$. Observe that for any $n$, if $\psi:\{1,\ldots,n\} \to A$ is any function then $d^S_{\sigma_n}(\phi,\psi)r \ge d^*_{\sigma_n}(\phi,\psi)$. So if $\epsilon>0$ then
$$\bE\Big[\big| \{\psi:\{1,\ldots,n\}\to A~:~d^S_{\sigma_n}(\phi,\psi)\le \epsilon\}\big|\Big]
\le\bE\Big[\big| \{\psi:\{1,\ldots,n\}\to A~:~d^*_{\sigma_n}(\phi,\psi)\le r\epsilon\}\big|\Big].$$

This implies $h(\Sigma,T,\phi) \le F(T,\phi).$ 


Recall that $B(e,n)$ denotes the ball of radius $n$ in $G$ and $f(T,\phi) = \inf_n F(T,\phi^{B(e,n)})$. Thus we have $\inf_n h(\Sigma,T,\phi^{B(e,n)}) \le f(T,\phi)$. Since $\phi$ and $\phi^{B(e,n)}$ generate the same $\sigma$-algebra, theorem \ref{thm:Bo08b} implies that $h(\Sigma,T,\phi)=h(\Sigma,T,\phi^{B(e,n)})$ for all $n$. This implies the proposition.
\end{proof}

\begin{prop}
 $$h(\Sigma,T,\phi) \ge f(T,\phi).$$
\end{prop}

\begin{proof}
Given a finite set $K \subset G$, define
$$h(\Sigma,T,\phi;K):=\inf_{\epsilon>0} \limsup_{n\to\infty} \frac{1}{n}\log\Big(\bE\Big[\big| \{\psi:\{1,\ldots,n\}\to A~:~d^K_{\sigma_n}(\phi,\psi)\le \epsilon\}\big|\Big]\Big).$$
{\bf Claim 1}. $h\big(\Sigma,T,\phi;B(e,m)\big) \ge F\big(T,\phi^{B(e,m)}\big)$ for all $m\ge 0$. 

Note that if $K \subset L$ then $h(\Sigma,T,\phi;K) \ge h(\Sigma,T,\phi;L)$. It follows that $h(\Sigma,T,\phi) = \inf_m h(\Sigma,T,\phi;B(e,m))$. Thus claim 1 implies the proposition.

To simplify notation, let $\B$ denote $B(e,m)$. To prove claim 1, for $m,n,\epsilon\ge 0$, let $P(m,n,\epsilon)$ be the set of all pairs $(\sigma,\omega)$ with $\sigma:G \to \sym(n)$ a homomorphism and $\omega:\{1,\ldots,n\} \to A$ a map such that $d^{\B}_\sigma(\phi,\omega)\le \epsilon$. Since there are $n!^r$ homomorphisms from $G$ into $\sym(n)$,
\begin{eqnarray}\label{eqn:P}
h(\Sigma,T,\phi;\B) = \inf_{\epsilon>0} \limsup_{n\to\infty} \frac{1}{n}\log \Big(\frac{|P(m,n,\epsilon)|}{n!^r}\Big) .
\end{eqnarray}
Let $Q(m,n,\epsilon)$ be the set of all pairs $(\sigma,\psi)$ with $\sigma:G \to \sym(n)$ a homomorphism and $\psi:\{1,\ldots,n\} \to A^{\B}$ a map such that $d^{*}_\sigma(\phi^{\B},\psi)\le \epsilon$. By theorem \ref{thm:second},
\begin{eqnarray}\label{eqn:Q}
F(T,\phi^{\B}) =  \inf_{\epsilon>0} \limsup_{n\to\infty} \frac{1}{n}\log \Big(\frac{|Q(m,n,\epsilon)|}{n!^r} \Big).
\end{eqnarray}

For $g\in \B$ let $\pi_g:A^{\B} \to A$ denote the projection map $\pi_g( (a_h)_{h\in \B} )=a_g$. For $(\sigma,\psi)\in Q(m,n,\epsilon)$, let $R(\sigma,\psi) = (\sigma, \pi_e \circ \psi)$. Let $H(x):=-x\log(x)-(1-x)\log(1-x)$.

{\bf Claim 2}. If $c=1+|\B|$ then the image of $R$ is contained in $P(m,n,\epsilon c)$.

{\bf Claim 3}. There are constants $C, k >0$ depending only on $m$ such that if $\epsilon<\frac{1}{4|B|}$ then $R$ is at most $C\exp(nk\epsilon + nH(2|\B|\epsilon))$ to 1. I.e., for any $(\sigma,\omega)$ in the image of $R$, $|R^{-1}(\sigma,\omega)| \le C\exp(nk\epsilon +nH(2|\B|\epsilon))$.

Claims 2 and 3 imply
$$ C\exp(kn\epsilon+nH(2|\B|\epsilon))\Big|P\big(m,n,\epsilon c)\big)\Big| \ge  \big| Q(m,n,\epsilon) \big|.$$
Together with equations \ref{eqn:P} and \ref{eqn:Q}, this implies claim 1 and hence the proposition.

Next we prove claim 2. For this purpose, fix a homomorphism $\sigma:G \to \sym(n)$. Observe that for any $x \in X$ and any $t\in \{s_1,\ldots,s_r\}$,
\begin{eqnarray*}
\pi_g\phi^{\B}(x) = \phi(T_gx)= \pi_{gt^{-1}}\phi^{\B}(T_tx) ~\forall g\in \B \cap \B t.
\end{eqnarray*}
Therefore if $i \in \{1,\ldots,n\}$ and for some $g\in \B \cap \B t$, $\psi:\{1,\ldots,n\}\to A^\B$ satisfies 
$$\pi_g\psi(i) \ne  \pi_{gt^{-1}}\psi\big(\sigma(t)i\big)$$
then $\psi\vee \psi^t(i) \ne \phi^{\B} \vee \phi^{\B t}(x)$ for any $x\in X$. 

So, let $\cG$ be the set of all $i\in \{1,\ldots,n\}$ such that for all $t \in \{s_1,\ldots, s_r\}$,
\begin{eqnarray*}
\pi_g\psi(i) &=& \pi_{gt^{-1}}\psi\big(\sigma(t)i\big) ~\forall g\in \B \cap \B t.
\end{eqnarray*}
Thus
$$d^*_\sigma(\phi^{\B},\psi) \ge \frac{|\cG^c|}{n}=\zeta(\cG^c)$$
where $\cG^c$ denotes the complement of $\cG$ and $\zeta$ denotes the uniform probability measure on $\{1,\ldots,n\}$.

Let $\cG_m$ be the set of all $i\in \{1,\ldots,n\}$ such that $\sigma(g)i \in \cG$ for all $g\in \B$. Note 
\begin{eqnarray}\label{note}
\zeta(\cG_m^c) \le \big|\B\big| \zeta(\cG^c) \le \big|\B\big|d^*_\sigma(\phi^{\B},\psi).
\end{eqnarray}

If $i \in \cG_m$ then $\psi(i)=(\pi_e\circ\psi)^{\B}(i)$. Therefore
\begin{eqnarray*}
\sum_{a\in A^{\B}} \Big|\psi_*\zeta(a)-(\pi_e\circ\psi)^{\B}_*\zeta(a)\Big| \le\frac{1}{n}\Big|\big\{i~:~\psi(i)\ne(\pi_e\circ\psi)^{\B}(i)\big\}\Big|\le \zeta(\cG_m^c)\le  |\B|d^*_\sigma\big(\phi^{\B},\psi\big).
\end{eqnarray*}

Suppose $d^*_\sigma(\phi^{\B},\psi)\le \epsilon$. Then
\begin{eqnarray*}
d^{\B}_\sigma(\phi,\pi_e \circ \psi) &=& \sum_{a\in A^{\B}} \big|\phi^{\B}_*\mu(a)-(\pi_e\circ\psi)^{\B}_*\zeta(a)\big|\\
&\le& \sum_{a\in A^{\B}} \big|\phi^{\B}_*\mu(a)-\psi_*\zeta(a)\big|+\big|\psi_*\zeta(a)-(\pi_e\circ\psi)^{\B}_*\zeta(a)\big| \\
&\le& d^*_\sigma(\phi^{\B},\psi)\big(1 + |\B|\big) \le \epsilon\big(1 + |\B|\big).
\end{eqnarray*}
This proves claim 2.

Let $(\sigma,\omega)$ be in the image of $R$. 

{\bf Claim 4}. For every $\psi$ with $R(\sigma,\psi)=(\sigma,\omega)$, there exists a set $L(\psi)\subset \{1,\ldots,n\}$ of cardinality $\lfloor n(1-|\B|\epsilon) \rfloor$ such that $\psi(i)=\omega^B(i)$ for all $i\in L(\psi)$. 

To prove claim 4, observe that, if $\cG_m$ is defined as above, then for all $i\in \cG_m$,  $\psi(i)=\omega^B(i)$. By equation \ref{note},
$$|\cG_m| =n\Big(1-\zeta(\cG_m^c)\Big) \ge n\Big(1-|B|d_\sigma^*\big(\phi^B,\psi\big)\Big) \ge n\big(1-|B|\epsilon\big).$$
So let $L(\psi)$ be any subset of $\cG_m$ with cardinality $\lfloor n(1-|\B|\epsilon) \rfloor$. This proves claim 4.

Next we prove claim 3. Claim 4 implies
\begin{eqnarray}\label{eqn:stirling}
|R^{-1}(\sigma,\omega)| \le |A|^{|\B|\big(n- \lfloor n(1-|\B|\epsilon) \rfloor\big)}{n \choose \big\lfloor n(1-|\B|\epsilon) \big\rfloor}.
\end{eqnarray}

This is because there are ${n \choose \lfloor n(1-|\B|\epsilon) \rfloor}$ sets in $\{1,\ldots,n\}$ with cardinality equal to $\lfloor n(1-|\B|\epsilon) \rfloor$ and for each $i\in \{1,\ldots,n\} -L(\psi)$, there are at most $|A|^{|\B|}$ possible values for $\psi(i)$.

Because $H$ is monotone increasing for $0<x<1/2$ it follows from Stirling's approximation that if $\epsilon<\frac{1}{4|\B|}$ then 
  $${n \choose  \big\lfloor n(1-|\B|\epsilon) \big\rfloor} \le C\exp\Big( n H\big(2|\B|\epsilon\big) \Big)$$
 where $C>0$ is a constant. This and equation \ref{eqn:stirling} now imply claim 3 and hence the proposition.
\end{proof}




\begin{thebibliography}{100000}




\bibitem[Bo08a]{Bo08a} L. Bowen. \textit{A measure-conjugacy invariant for actions of free groups}. arXiv:0802.4294, to appear in the Annals of Mathematics.

\bibitem[Bo08b]{Bo08b} L. Bowen. \textit{Measure conjugacy invariants for actions of countable sofic groups}. arXiv:0804.3582, to appear in the Journal of the A.M.S.


\bibitem[Bo08c]{Bo08c} L. Bowen. \textit{Nonabelian free group actions: Markov processes, the Abramov-Rohlin formula and Yuzvinskii's formula}. arXiv:0806.4420








\bibitem[FO70]{FO70} N. A. Friedman and D. S. Ornstein. \textit{On isomorphism of weak Bernoulli transformations}.  Advances in Math.  5  (1970) 365--394.












\bibitem[Or70]{Or70} D. Ornstein. \textit{Factors of Bernoulli shifts are Bernoulli shifts}.  Advances in Math.  5  1970 349--364 (1970).




\bibitem[Po08]{Po08} S. Popa. private communication. 

\bibitem[PS07]{PS07} S. Popa and R. Sasyk. \textit{On the cohomology of Bernoulli actions}.  Ergodic Theory Dynam. Systems  27  (2007),  no. 1, 241--251.










\end{thebibliography}
\end{document}